\documentclass[11pt]{amsart}
\usepackage{geometry}                
\usepackage{graphicx}
\usepackage{amssymb}
\usepackage{bbm}
\usepackage{epstopdf}
\usepackage{amsthm} 
\usepackage{marvosym}
\usepackage{enumerate}
\usepackage[all]{xy}
\usepackage[linktocpage]{hyperref}

\newcommand{\R}{\mathbb{R}} 
\newcommand{\Z}{\mathbb{Z}}  
\DeclareMathOperator{\lk}{lk}
\DeclareMathOperator{\st}{st}
\DeclareMathOperator{\Sing}{split}
\DeclareMathOperator{\id}{id}
\DeclareMathOperator{\GL}{GL}
\DeclareMathOperator{\SL}{SL}
\DeclareMathOperator{\SO}{SO}
\DeclareMathOperator{\Aut}{Aut}
\DeclareMathOperator{\Out}{Out}
\DeclareMathOperator{\Min}{Min}

\DeclareMathOperator{\Isom}{Isom}

\DeclareMathOperator{\Int}{Int}

\newcommand{\G}{\Gamma}             
\newcommand{\AG}{A_\G}                 
\newcommand{\WP}{{\mathcal P}}     
\newcommand{\WQ}{{\mathcal Q}}     

\newcommand{\bPi}{\Pi}  

\newcommand{\SG}{{\Sigma}_\Gamma}                    
\newcommand{\tS}{{\mathcal{T}}_\Gamma}              
\newcommand{\OG}{\mathcal{O}_{\Gamma}}            
\newcommand{\SaG}{\mathbb{S}_{\Gamma}}           
\newcommand{\Sa}{\mathbb{S}}                                 
\newcommand{\SP}{\Sa^\bPi}                                     
\newcommand{\USP}{\widetilde\Sa^\bPi}                     
\newcommand{\F}{\mathcal{F}}                                    
\newcommand{\E}{\mathcal{E}}                                    
\newcommand{\EP}{\mathbb{C}^\bPi}			 

\newcommand{\re}{\mathbbm r} 

\newtheorem{thm}{Theorem}

\newtheorem{lemma}[thm]{Lemma}

\newtheorem{proposition}[thm]{Proposition}

\newtheorem{corollary}[thm]{Corollary}
\numberwithin{thm}{section}

\theoremstyle{remark}
\newtheorem{rmk}{Remark}

\theoremstyle{definition}
\newtheorem{definition}[thm]{Definition}

\usepackage{xcolor}
\definecolor{darkorange}{rgb}{1.0, 0.55, 0.0}
\definecolor{forestgreen}{rgb}{0.13, 0.55, 0.13}

\newcommand{\overl}[1]{\overline{#1}}

\newcommand{\mbb}{\mathbb}

\newcommand{\T}{\mathbb{T}}
\newcommand{\K}{\mathbb{K}}
\newcommand{\Cent}{\mathcal{Z}}

\title{Isometry groups of skewed $\G$-complexes}
\author{Corey Bregman}

\begin{document}
\maketitle

\begin{abstract} Let $\AG$ be a right-angled Artin group. Charney, Vogtmann and the author constructed an outer space for $\Out(\AG)$ generalizing both $CV_n$ for $\Out(F_n)$ and the symmetric space $\SL_n(\R)/\SO_n(\R)$ for $\GL_n(\Z)$. Points in this space are equivalence classes of pairs $(X,\rho)$ where $\rho\colon X\rightarrow \SaG$ is a homotopy equivalence from $X$ to the Salvetti complex $\SaG$ and $X$ is a locally CAT(0) space called a skewed  $\G$-complex. In this note we show that any isometry of a skewed  $\G$-complex which is homotopic to the identity lies in the identity component of $\Isom(X)$.  As a corollary, we prove that the group of path components of $\Isom(X)$ is finite and  injects into $\Out(\AG)$.
\end{abstract}

\section{Introduction}

Let $\G$ be a finite simplicial graph, and let $\AG$ be the associated right-angled Artin group.  Right-angled Artin groups (RAAGs) generalize both free groups $F_n$ and free abelian groups $\Z^n$, and sit inside the larger family of Artin groups. Their relative importance in geometric group theory in recent years comes from their close connection with CAT(0) cube complexes, which has been exploited to solve several problems in low-dimensional topology \cite{Sageev, BeWi, HaWi, AGM}.

Automorphism groups of RAAGs have been the subject of many investigations, though primarily by algebraic means \cite{ChVo1, ChVo2, Day, DayWade, DaySaleWade}. In contrast, one of the most powerful tools in the study of automorphisms of the free group is the action of $\Out(F_n)$ on Culler--Vogtmann outer space $CV_n$ \cite{CuVo}.  In an effort to bring similar geometric techniques  to bear on automorphisms of general RAAGs, Charney, Vogtmann, and the author constructed a finite-dimensional, contractible outer space $\OG$ on which $\Out(\AG)$ acts with finite stabilizers \cite{BCV}. The construction of $\OG$, equipped with the action of $\Out(\AG)$, generalizes that of  $CV_n$ for $\Out(F_n)$ and the classical symmetric space $\SL_n(\R)/\SO_n(\R)$ for $\GL_n(\Z)$.

Points in $\OG$ are equivalence classes of pairs $(X,\rho)$ where $X$ is a locally CAT(0) metric space called a skewed $\G$-complex and the \emph{marking} $\rho\colon X\rightarrow \SaG$ is a homotopy equivalence to the Salvetti complex $\SaG$.  Two pairs $(X,\rho)$ and $(X',\rho')$ are equivalent if there exists an isometry $i\colon X\rightarrow X'$ that commutes with the markings up to homotopy, \emph{i.e.}, such that $\rho\simeq \rho'\circ i$. As a consequence, understanding the isometry group of the skewed $\G$-complex $X$ is a crucial ingredient in the construction of $\OG$ and the action of $\Out(\AG)$ on it.  The structure of this isometry group is the subject of this short note.  

$\G$-complexes, first introduced in \cite{CSV}, are a family of nonpositively curved cube complexes with fundamental group $\AG$.  Every $\G$-complex $X$ comes with multiple collapse maps to the Salvetti complex  $\SaG$.  Each collapse map is a homotopy equivalence $c\colon X\rightarrow \SaG$ obtained by identifying points in carriers of hyperplanes of $X$ (one thinks of the collapsed subcomplex as analogous to a maximal tree of a graph). A collapse map identifies $\pi_1(X)$ with $\pi_1(\SaG)=\AG$, and distinct collapse maps result in identifications that differ by  automorphisms of $\AG$.

Importantly, not all automorphisms of $\AG$ arise in this way, but only the so-called \emph{untwisted automorphisms} $U(\AG)$.  It was this fact that allowed Charney, Stambaugh and Vogtmann to define an outer space $\SG$ for $U(\AG)$ in \cite{CSV}. Untwisted automorphisms behave more like automorphisms of $F_n$ than automorphisms of $\Z^n$. For this reason, $\SG$ is similar in construction to $CV_n$.  In particular, they prove that $\SG$ also admits an equivariant retraction to a cocompact spine $K_\G$. Points in $K_\G$  correspond to pairs $(X,h)$ where $X$ is a $\G$-complex and $h$ is an untwisted marking, \emph{i.e.}, a homotopy equivalence $h\colon X\rightarrow \SaG$ such that $h\circ c^{-1}$ lies in $U(\AG)$, where $c^{-1}$ is a homotopy inverse of some (and hence any) collapse map $c$.

The construction of $\OG$ relies in a fundamental way on Charney--Stambaugh--Vogtmann's construction. To pass from $K_\G$ to $\SG$, one uses rectilinear cubes with variable side-lengths in the construction of the $\G$-complex $X$, rather than just unit cubes $[0,1]^n$.  To pass from $\SG$ to $\OG$, rectilinear cubes are then replaced with euclidean parallelotopes, subject to certain restrictions.  A family of parallelotopes that satisfy these restrictions is called an \emph{allowable} in \cite{BCV}, and the resulting metric is called an \emph{allowable metric}.  A \emph{skewed $\G$-complex} is a $\G$-complex $X$ with an allowable metric $d$.

Let $\Isom(X)$ be the isometry group of a skewed $\G$-complex $(X,d)$.  Equipped with the metric topology, $\Isom(X)$ is a locally compact topological group.  The connected component of the identity, denoted $\Isom_0(X)$, consists of isometries that are isotopic, \emph{through isometries}, to the identity. A fortiori, elements of $\Isom_0(X)$ are therefore homotopic to the identity.  Conversely, we prove

\begin{thm}\label{thm:main} Let $X$ be a skewed $\G$-complex and let $f\colon X\rightarrow X$ be an isometry. If is homotopic to the identity, then $f\in \Isom_0(X)$.  
\end{thm}
Let $\pi_0(\Isom(X))=\Isom(X)/\Isom_0(X)$ denote the group of path components of $\Isom(X)$. It follows from general theory of CAT(0) spaces (see Proposition \ref{prop:FiniteComps} below) that $\pi_0(\Isom(X))$ is finite.  As a corollary of Theorem \ref{thm:main} we  obtain
\begin{corollary}\label{cor:main} $\pi_0(\Isom(X))$ is finite and the natural map $\Isom(X)\rightarrow \Out(\AG)$ induces an injection $\pi_0(\Isom(X))\rightarrow \Out(\AG)$.
\end{corollary}
\vspace{1em}

\subsection{Relation to \cite{BCV}} In Section 7 of \cite{BCV} it is shown that $\OG$ is contractible.  The strategy is to construct a contractible space $\tS$ together with a fibration $\Theta\colon \tS\rightarrow \OG$ whose fibers are contractible.  Each element of the fiber is isometric to a fixed marked, skewed $\G$-complex $(X,\rho)$. We choose a basepoint in the fiber and describe each other point in terms of certain ``shearing coordinates" relative to the basepoint.   These shearing coordinates describe a real vector space, hence the fibers are contractible. 

Theorem \ref{thm:main} allows us to establish that the shearing coordinates are well-defined. In order to define the shearing coordinates we first pick an isometry from the point in the fiber to $(X,\rho)$.  Although this isometry need not be unique, the fact that points in $\OG$ are marked means that it is unique \emph{up to homotopy}.  Therefore, by Theorem \ref{thm:main}, any two choices of isometry differ by an element of $\Isom_0(X)$. One then proves that $\Isom_0(X)$ leaves the shearing coordinates unchanged.

\subsection{Outline of the proof.} In \S\ref{sec:Gcomplexes}, we review the relevant properties $\G$-complexes and skewed metrics that we will need. We prove that every skewed $\G$-complex $X$ has the geodesic extension property, and use this to identify $\Isom_0(X)$ with the group of isometries generated by translating along the central torus of $X$. 

 In \S\ref{sec:MaxTori}, we study locally convex subspaces of  $X$ called  \emph{maximal tori} that are isometric to tori with flat metrics. Maximal tori can be characterized in terms of maximal abelian special subgroups of $\AG$.  We show that every point of $X$ lies in at least one maximal torus, and that any two maximal tori intersect in a locally convex subspace which is homotopy equivalent to a torus representing the intersection of the corresponding maximal abelian  special subgroups.  Moreover, any two such tori can be connected by a chain of maximal tori where consecutive tori have non-empty intersection.  
 
 Finally in \S\ref{sec:ProofofMain} we prove Theorem \ref{thm:main}. We sketch the proof here.  Any isometry which is homotopic to the identity must restrict to a translation on each maximal torus.  The fact that the isometry must preserve intersections of maximal tori further restricts these translations to lie in the torus corresponding to the intersection of the abelian subgroups. Since the intersection of all maximal abelian special subgroups is just the center of $\AG$, the translation must be in the central torus.  Because the maximal tori cover $X$, we conclude that the isometry must be a translation of the central torus of $X$.

 \vspace{1em}
\textbf{Acknowledgements.} I am grateful to my collaborators, Ruth Charney and Karen Vogtmann, for their help in writing this paper.  The topic of this paper was what sparked our initial collaboration 5 years ago, which later turned into \cite{BCV}. It is fitting then, that this work should fill in a small but necessary step in that project.  I would also like to thank Ruth specifically for her constant encouragement and mentorship over the years.

\section{Background on skewed $\G$-complexes}\label{sec:Gcomplexes}
This section will be a review of facts concerning automorphisms of RAAGs, blowups and $\G$-complexes, and skewed metrics on $\G$-complexes.  We will assume familiarity with CAT(0) spaces and cube complexes.  Full details for the material quoted in this section can be found in \cite{CSV} and in Sections 2 through 5 of \cite{BCV}.

\subsection{RAAGs and their automorphisms}
Let $\Gamma=(V,E)$ be a finite simple graph.  The right-angled Artin group (RAAG) associated to  $\G$ is the group $\AG$ with presentation \begin{equation}\label{eq:Presentation}\AG=\langle v\in V\mid [v,w] \text{ if } \{v,w\}\in E\rangle\end{equation}
Any induced subgraph $\Delta\subseteq \Gamma$ generates a subgroup $A_\Delta$ called a \emph{special subgroup}. A \emph{clique} $\Delta\subseteq \Gamma$ is a complete subgraph of $\Gamma$.  $\Delta$ is \emph{maximal} if it is not contained in a larger clique. Maximal cliques in $\Gamma$ correspond to maximal abelian special subgroups of $A_\Gamma$. We denote by $A_\Delta$ the special subgroup generated by the vertices in $\Delta$. Let $\Cent(\Gamma)$ be the intersection of all maximal cliques in $\Gamma$.  The special subgroup $A_{\Cent(\Gamma)}$ is just the center $\Cent(A_\Gamma)$.

For any $v\in V$, the link $\lk(v)$ is the induced subgraph on vertices adjacent to $v$. The link consists of those generators which are distinct from $v$ but which commute with $v$.   The star $\st(v)=\{v\}\cup \lk(v)$, and $A_{\st(v)}$ is the centralizer of $v$ in $\AG$. We define two partial orderings on $V$ as follows. The \emph{fold ordering} is defined as $v\leq_w$ if and only if $\lk(v)\subseteq \lk(w)$.    The twist ordering is defined as $v\leq_t w$ if and only if $\st(v)\subseteq \st(w)$.  If $\lk(v)=\lk(w)$ (resp. $\st(v)=\st(w)$) we call $v$ and $w$ \emph{fold-equivalent} (resp. \emph{twist-equivalent}). 

A natural set of generators for $\Aut(\AG)$ and hence $\Out(\AG)$ can be defined in terms of $\G$ and the twist and fold relations:
\begin{definition}The \emph{Laurence--Servatius generators} for $\Aut(\AG)$ consist of the following four types of automorphisms
\begin{enumerate}
\item \textbf{Graph Automorphisms:} elements of $\Aut(\G)$ induce automorphisms of $\AG$.
\item \textbf{Inversions:} For each $v\in V$, $i_v\colon v\mapsto v^{-1}$ and acts as the identity elsewhere. 
\item \textbf{Partial Conjugations:} For any $v\in V$, let $C$ be a connected component of $\Gamma\setminus \st(v)$.  The partial conjugation $\chi_{v,C}\colon w\mapsto vwv^{-1}$ if $w\in C$, and is the identity otherwise.
\item  \textbf{Transvections:} If $\lk(w)\subseteq \st(v)$, then $\phi_{v,w}\colon w\rightarrow vw$ and acts as the identity otherwise.  If $w\leq_t v$, then $\phi_{v,w}$ is a \emph{twist}.  If $w\leq_f v$, then $\phi_{v,w}$ is a \emph{fold}.  
\end{enumerate}
\end{definition}
Clearly conjugating $\AG$ by $v$ is the same as partially conjugating each component of $\G\setminus \st(v)$, hence the product of each of these partial conjugations is trivial in $\Out(\AG)$. The subgroup generated by graph automorphisms, inversions, partial conjugations, and folds is called the \emph{untwisted subgroup} $U(\AG)$. In \cite{CSV}, Charney--Stambaugh--Vogtmann construct a finite-dimensional contractible space $K_\G$ on which $U(\AG)$ acts properly and cocompactly. 

The goal of \cite{BCV} was therefore to build a space to incorporate the action of twists. We will say that $v\in V$ is \emph{twist-dominant} if there exists $w\in V$ such that $w\leq_tv$.  Equivalently, $v$ is twist-dominant if there exists a twist $\phi_{v,w}\colon w\rightarrow vw$. On the other hand, the next proposition, which may be found in Section 2 of \cite{BCV}, follows from the definitions 
\begin{proposition}\label{prop:TDNF}
If $v$ is twist-dominant, there does not exist $u\neq v$ such that $v\leq_f u$.  
\end{proposition}

In particular, if $v$ is twist-dominant then it is not fold-equivalent to any other generator. This proposition is a simple observation, but plays a crucial role in the definition of skewed metrics for $\G$-complexes and the construction of $\OG$ more generally.

\subsection{Blowups and $\G$-complexes}
Throughout, a $n$-cube will be the space isometric to $[0,1]^n$.  A cube complex is a metric space obtained by gluing cubes along their faces by isometries, where the metric is the induced path metric.  A cube complex is called \emph{nonpositively curved} (NPC) if it satisfies Gromov's link condition at vertices: the link of each vertex is a flag simplicial complex.  If $X$ is an NPC cube complex, the universal cover of $X$ is CAT(0), and in particular, $X$ is aspherical.

The Salvetti complex $\SaG$ is a nonpositively curved (NPC) cube complex $\SaG$ whose fundamental group is isomorphic to $\AG$.  The 2-skeleton of $\SaG$ is just the presentation complex for $\AG$ with the presentation in Equation \ref{eq:Presentation}. That is, we have a wedge of $|V|$ circles,  and each relation adds a 2-cube along $vwv^{-1}w^{-1}$ whenever $\{v,w\}\in E$. For each $n$-clique in $\G$, we then add a $n$-cube in the $n$-skeleton of $\SaG$.  By construction, $\SaG$ satisfies Gromov's link conditon and thus is nonpositively curved.  

In \cite{BCV}, Charney, Stambaugh and Vogtmann build a family of NPC cube complexes called blowups. Each blowup has fundamental group $\AG$ and is associated to a collection of automorphisms called $\Gamma$-Whitehead automorphisms, which we now describe.

\begin{definition} Let $V^\pm=V\cup V^{-1}$ be the union of all generators and their inverses.  Given $m\in V$,  a \emph{$\G$-Whitehead partition $\WP$ based at $m$} is a partition of $V^\pm$ into three sets $P^+$, $P^-$ (called the sides of $\WP$) and $\lk(\WP)$.  
\begin{enumerate}[(i)]
\item $\lk(\WP)$ consists of $\lk(m)\cup \lk(m)^{-1}$
\item $m$ and $m^{-1}$ lie in different sides of $\WP$, but neither $P^+$ nor $P^-$ is a singleton
\item If $v\neq w$ are in the same component of $\Gamma\setminus \st(m)$, then $v^\pm,w^\pm$ are in the same side of $\WP$.
\end{enumerate}

\end{definition}
If $v$ and $v^{-1}$ lie in different sides of $\WP$, we say $\WP$ splits $v$. By (ii), $m\in \Sing(\WP)$ and by (iii), any other $v\in \Sing(\WP)$ satisfies $v\leq_f m$.  Thus, each side of $\WP$ consists of a union of components (with inverses) of $\Gamma\setminus \st(m)$, or singletons $v$ or $v^{-1}$ such that $v\leq_f m$.   We denote by  $\Sing(\WP)$ be the set of $v$ that are split by $\WP$, and set $\max(\WP)\subset \Sing(\WP)$ to be those generators fold-equivalent to $m$. The partition $\WP$ defines a \emph{$\G$-Whitehead automorphism} $\phi(\WP,m)$ that inverts $m$, acts as a fold on elements of $\Sing(\WP)\setminus \{m\}$, a partial conjugation by $m$ on $P^-$ and the identity on $P^+\cup \lk(\WP)$ (see \cite{BCV} for details.)

One important consequence of Proposition \ref{prop:TDNF} is 
\begin{lemma}\label{lem:TDMax}If $v$ is twist-dominant and $v\in \Sing(\WP)$, then $\max(\WP)=\{v^\pm\}$. 
\end{lemma}
\noindent
In other words, the only way a partition $\WP$ can split a twist-dominant generator $v$ is if $\WP$ is based at $v$.

Two $\G$-Whitehead partitions $\WP, \WQ$ are called \emph{adjacent} if every element of $\max(\WP)$ commutes with every element of $\max(\WQ)$.  This is well-defined since every element of $\max(\WP)$ is fold-equivalent.  We also say $v\in V^\pm$ and $\WP$ are adjacent  if $v\in \lk(\WP)$.  Two partitions $\WP, \WQ$ are called \emph{compatible} if  they are adjacent  or if $P^\times\cap Q^\times=\emptyset$ for some choice of sides $P^\times, Q^\times$.In fact, if $\WP$ and $\WQ$ are compatible but not adjacent, then  $P^\times\cap Q^\times=\emptyset$ for exactly one choice of sides $P^\times, Q^\times$. It is this observation that allows for the construction of a blowup for each collection of pairwise compatible partitions.   

Let $\bPi=\{\WP_1,\ldots, \WP_k\}$ be a collection of pairwise compatible partitions.  A \emph{region} is a choice of sides $P_i^\times$ such that whenever $\WP_i, \WP_j$ are not adjacent, $P_i^\times\cap P^\times_j$ is nonempty.  For any $v\in V^\pm$, either $v\in \lk(\WP_i)$ or is contained in exactly one side of $\WP_i$.  A region $\re=\{P_1^\times,\ldots, P_k^\times\}$ is called a \emph{terminal region for $v$} if contains of all sides containing $v$. If $\re$ is a terminal region for $v$, by switching sides along all the partitions containing $v$, we obtain a terminal region $\re^{*v}$ for $v^{-1}$ We now describe the blowup $\SP$ associated to $\bPi$.  

\begin{definition} The \emph{blowup} associated to a compatible collection of partitions $\bPi=\{\WP_1,\ldots, \WP_k\}$ is the cube complex $\SP$ constructed as follows.
\begin{enumerate}[(i)]
\item The vertices of $\SP$ are in one-to-one correspondence with regions $\re$
\item Each edge is labeled by either $v\in V$ or by a partition $\WP_i\in \bPi$. We add an edge labeled $\WP_i$ between two regions if they differ only by changing sides of $\WP_i$. If $\re$ is a terminal region for $v$, then we add an edge labeled $v$ from $\re$ to $\re^{*v}$.
\item An $n$-dimensional cube is attached in whenever there exists a collection of pairwise adjacent edge labels $\{A_1,\ldots, A_n\}$ at a vertex.
\end{enumerate}
\end{definition}

It follows from the definition that each edge labeled by $A\in V\cup \bPi$ is dual to a unique hyperplane $H_A$.  We can extend the twist and fold orderings on $V$ to $V\cup \bPi$ in terms of maximal elements. Recall that all elements in $\max(\WP)$ are fold equivalent and therefore have the same link. Hence, the fold and twist relation depend only on the fold equivalence class of a generator. For any $A, B\in V\cup \bPi$, define $A\leq_f B$ (resp. $A\leq_t B$) to mean $\max(A)\leq \max(B)$ (resp. $\max(A)\leq_t \max(B)$ where by convention we set $\max(v)$ to be $v$ if $v\in V$.

Edges labeled by $\WP$ join distinct regions and thus are always embedded, hence the hyperplane carrier $\kappa(H_\WP)$ is embedded as well.  We can therefore collapse along $H_\WP$ to get a homotopy equivalent cube complex with one fewer hyperplane.  Theorem 4.6 of \cite{CSV} states that the cube complex that results from collapsing $H_\WP$ in $\SP$ is isomorphic to $\Sa^{\bPi-\WP}$. In particular, collapsing the hyperplanes labeled by each $\WP_i\in \bPi$ we obtain a homotopy equivalence called the \emph{standard collapse map} $c_\pi\colon \SP\rightarrow \Sa^\emptyset=\SaG$.

The following facts about blowups will be needed in the remainder of this paper.  Their proofs can be found in either \cite{CSV} or Section 3 of \cite{BCV}.

\begin{proposition}\label{prop:BlowupFacts} Let $\SP$ be a blowup. 
\begin{enumerate}[(i)]
\item $\SP$ is NPC and the standard collapse map $c_\pi\colon \SP\rightarrow \SaG$ is a homotopy equivalence. 
\item The subcomplex $\EP\subset \SP$ consisting of all cubes with edge labels in $\bPi$ is CAT(0) and locally convex. 
\item For each maximal collection of pairwise adjacent edge labels in $V\cup \bPi$ there is a unique maximal cube in $\SP$.
\end{enumerate}
\end{proposition}

The complex $\EP$ plays a role in $\SP$ analogous to that of a maximal tree in a graph.  After collapsing a maximal tree one obtains a wedge of $n$ circles, \emph{i.e.}, the Salvetti complex for $F_n$.  Just as in a graph where there are many maximal trees, there are many subcomplexes in $\SP$ whose collapse results in $\SaG$.  Equivalently, each subcomplex corresponds to a different identification of the underlying cube complex with $\SP$.  

\begin{definition} A \emph{ $\G$-complex} is any cube complex combinatorially isomorphic to $\SP$ for some $\bPi$.
\end{definition}
Thus we will refer to a cube complex as a $\G$-complex whenever there is no preferred identification with $\SP$, or equivalently, no preferred collapse map.  

\subsection{Characteristic cycles}
Following \cite{BCV} we define a \emph{characteristic cycle for v} to be an edge-path $\chi_v\subset \SP$  comprised of any $e_v$ edge followed by a minimal length edge-path between its endpoints in $\EP$. The edges appearing in $\chi_v\cap \EP$ are labeled by exactly those partitions that split $v$. By construction, $\chi_v$ maps onto the loop labeled by $v$ under the collapse map $c_\pi\colon \SP\rightarrow \SaG$.  Thus, $\chi_v$ is homotopic to a locally geodesic loop $\beta_v$ representing $v$, which is the image of an axis for $v$. 

Let $p\colon \USP\rightarrow \SP$ be the universal covering projection. Then $\AG$ acts on $\USP$ by semisimple isometries and for any $g\in \AG$, let $\Min(g)$ be the min set for the action of $g$ on  $\USP$. We will need the following lemma, the proof of which can be found in \cite{BCV}. 
\begin{lemma}[\cite{BCV}, Lemma 3.10]\label{lem:BCV3.9}Let $A\in V\cup \Pi$ be a label, and suppose $v\in \max(A)$.  
\begin{enumerate}[(i)]
\item For each edge labeled $A$ there is a unique edge $e_v$ such that $e_A$ and $e_v$ are contained in a local geodesic $\beta_v$, hence every characteristic cycle for $v$ containing $e_v$ contains $e_A$, and vice versa.
\item The carrier $\kappa(H_A)$ is contained in the image of $\Min(v)$ and the induced cubical structures on $H_A$ and $H_v$ are isomorphic.  
\end{enumerate}
\end{lemma}

In particular, in the twist-dominant case, by Lemma \ref{lem:TDMax} we have
\begin{corollary}\label{cor:TDGeodesic}Let $v$ be twist-dominant.  Then every characteristic cycle for $v$ is the image of an axis for $v$, and the image of $Min(v)$ in $\SP$ is the union of all carriers $\kappa(H_A)$ such that $v\in \max(A)$.
\end{corollary}

\subsection{Skewed $\G$-complexes}A skewed metric on a blowup $\SP$ replaces cubes with general parallelotopes, subject to certain conditions which we call \emph{allowable}. Parallelotopes are identified in the same way as in $\SP$, and faces must still be glued by isometries. An $n$-dimensional parallelotope is the image of $[0,1]^n$ under an element of $\GL_n(\R)$.  A parallelotope in which all edges are orthogonal will be called an \emph{orthotope}. 

There is a convenient way to describe when a parallelotope is allowable in terms of which matrices in $\GL_n(\R)$ may be used to obtain  it. Given a blowup $\SP$, fix a total ordering $\preceq$ on the edge labels in $V\cup \bPi$ extending the twist ordering defined in the previous section. That is,  the total order $\preceq$ satisfies $A\leq_tB$ implies $A\preceq B$. 

Let $C$ be a maximal cube in $\SP$, with edge labels $\{A_1,\ldots, A_n\}$ such that $A_{i}\succeq A_{i+1}$ for all $i$.  We can think of $C$ as the unit cube in $\R^n$, where each edge labeled $A_i$ is parallel to the standard basis vector $e_i\in \R^n$.  Any matrix $M$ which preserves the ordering can be written as $M=DU$ where $U$ is an upper triangular unipotent matrix and $D$ is diagonal (and nonsingular).  Moreover, any parallelotope is isometric to a parallelotope obtained by such an $M$.  The image of $C$ under left multiplication by $M$ is allowable if and only if $U=(u_{ij})$ satisfies \[u_{ij}\neq 0 \text{ for $i<j$}\Rightarrow A_i\geq_tA_j\]  Note that in \cite{BCV}, we opted for a definition of allowable which does not rely on a choice of total ordering.  Thus, the description above is equivalent to that used in Section 5, but not what is stated there.  

\begin{definition}\label{def:Allowable} Let $\F$ be a choice of parallelotope metric for each cube in $\SP$ such that
\begin{enumerate}[(i)]
\item If $F$ is a face of $C$, the restriction of the metric from $C$ agrees with the parallelotope metric of $F$.
\item If $\max(A)=\{v\}$ is twist-dominant, then for any $B$ adjacent to
$A$, the angle between $e_A$ and $e_B$ is equal to the angle between $e_v$ and $e_B$.
\end{enumerate} We call $\F$ an \emph{allowable parallelotope structure}, and the resulting path metric on $\SP$ is  called an \emph{allowable metric}. 
\end{definition}

Note that condition (ii)  can always be arranged because of Corollary \ref{cor:TDGeodesic}, and implies that characteristic cycles for $v$ remain locally geodesic in any allowable metric.    A consequence of Definition \ref{def:Allowable} is that every allowable metric is locally CAT(0).  For a proof of this fact, and more details on the construction of allowable metrics, see Section 5 of \cite{BCV}.

\begin{definition} A \emph{skewed $\G$-complex} is any metric space isometric to $(\SP,d)$ where $d$ is an allowable metric. 
\end{definition}

Let $(\SP,\F)$ be an allowable parallelotope structure on $\SP$.  If every parallelotope in $\F$ is an orthotope, we call $\F$ \emph{rectilinear}.   In  Section 5 of \cite{BCV}, we define a \emph{straightening map}  $\sigma$ which replaces every allowable parallelotope structure $(\SP,\F)$ with a rectilinear parallelotope structure $(\SP,\sigma(\F))=(\SP,\E)$ possessing the same edge lengths.  In fact, we define a 1-parameter family of allowable parallelotope structures $\sigma_t(\F)$ such that $\sigma_0(\F)=\F$ and $\sigma_1(\F)=\E$. In terms of the matrices described above, straightening changes each entry $u_{ij}$ to 0 where $i\neq j$ is such that $A_i\geq_t A_j$.   For any $A\in V\cup \bPi$ and any $t$, the length of $e_A$ in $\sigma_t(\F)$ is the same as in $\F$. See section 5 of \cite{BCV} for details.

\subsection{The structure of the isometry group}
Let $(X,d)$ be a skewed $\G$-complex.  The isometries of $X$ form a topological group $\Isom(X)$ with the $C^0$-topology induced by $d$.  Denote by $\Isom_0(X)\leq \Isom(X)$ the connected component of the identity.  $\Isom_0(X)$ is a normal subgroup and the quotient $\Isom(X)/\Isom_0(X)$ is isomorphic to $\pi_0(\Isom(X))$, the group of path components of $\Isom(X)$.

The general structure of $\Isom(X)$ relies on the fact that skewed $\G$-complexes have the geodesic extension property, which we now prove. 

\begin{lemma}\label{lem:NoFreeFaces} $(X,d)$ has the geodesic extension property.
\end{lemma}
\begin{proof} Let $(\SP,\F)$ be an allowable parallelotope structure on $(X,d)$ so that $d=d_\F$. Since $\F$ has only finitely many isometry types of parallelotopes, the geodesic extension property is equivalent to not having free faces, according to  \cite[Proposition II.5.10]{BH}.  Let $C$ be a maximal parallelotope of $\SP$ with edge labels $\{A_1,\ldots, A_k\}$. By Lemma \ref{lem:BCV3.9}, each edge $e_{A_i}$ lies on a characteristic cycle $\chi_{v_i}$ for some $v_i\in \max(A_i)$. Since every edge label $B$ on $\chi_{v_i}$ satisfies $\max(B)\geq_f \max(A_i)$,  $B$ is adjacent to $A_j$ for every $j\neq i$. This implies that each face of $C$ is contained in the boundary of at least two maximal parallelotopes, hence is not free.
\end{proof}

Since $X$ has the geodesic extension property, it follows from Lemma II.6.16 of \cite{BH} that the min set of every element of $\Cent(\Gamma)$ is all of $X$.  By the flat torus theorem (Theorem II.7.1 ad loc.), $X$ decomposes as an orthogonal product $X=X_0\times \T_{\Cent(\Gamma)}$, where $\T_{\Cent(\Gamma)}$ is a torus of dimension equal to  $|\Cent(\Gamma)|$. We call $\T_{\Cent(\Gamma)}$ the \emph{central torus} of $X$.  Regarding $\T_{\Cent(\Gamma)}$ as a topological group, the action of $\T_{\Cent(\G)}$ on itself by translations induces a family of  isometries of $X$. Explicitly, given $\tau\in \T_{\Cent(\Gamma)}$, the map \[L_\tau\colon X_0\times \T_{\Cent(\G)}\rightarrow X_0\times \T_{\Cent(\G)}\] defined by $L_\tau(x,t)=(x,t+\tau)$ will be called the translation by $\tau$.  The collection of all $L_\tau$ generate an abelian subgroup of $\Isom(X)$ isomorphic to $\T_{\Cent(\G)}$ itself. Since every $\tau \in \T_{\Cent(\G)}$ lies on a 1-parameter subgroup, the translation subgroup of $\Isom(X)$ is contained in $\Isom_0(X)$. In fact, we have

\begin{proposition}\label{prop:FiniteComps}The group $\Isom(X)$ has finitely many connected components, and $\Isom_0(X)$ is the group of translations of $\T_{\Cent(\G)}$.
\end{proposition}
\begin{proof}  Theorem II.6.17 of \cite{BH} states that the isometry group of any compact, locally CAT(0) space with the geodesic extension property is a topological group with finitely many connected components.  Moreover, the connected component of the identity is the group of translations by the central torus.  The proposition now follows from Lemma \ref{lem:NoFreeFaces}.  
\end{proof}

Since $X$ is a $K(\AG,1)$, we can identify the group of self-homotopy equivalences of $X$ up to homotopy with $\Out(\AG)$. We then obtain a homomorphism \[\psi\colon \Isom(X)\rightarrow \Out(\AG)\]

Every element of $\Isom_0(X)$ is isotopic to the identity, so $\Isom_0(X)\leq \ker\psi$. Combined with Proposition \ref{prop:FiniteComps} and the preceding discussion, Theorem \ref{thm:main}  and Corollary \ref{cor:main} now easily follow from

\begin{thm}\label{thm:IsoID}$\ker\psi=\Isom_0(X)$.
\end{thm}

This theorem will be proven in the following two sections. Here we deduce a strengthening of Corollary \ref{cor:main}, by considering the action of $\Isom(X)$ on $H_1(X)$. 

\begin{corollary}Let $f\colon X\rightarrow X$ be an isometry such that $f_*\colon H_1(X)\rightarrow H_1(X)$ is trivial. Then $f\in \Isom_0(X)$.  In particular, $\pi_0(\Isom(X))$ injects into $\Aut(H_1(X))$.
\end{corollary}
\begin{proof}Let $\overl{\psi}\colon \Isom(X)\rightarrow \Aut(H_1(X))$ be the homomorphism induced by the action of $\Isom(X)$ on $H_1(X)$. Clearly, we have $\Isom_0(X)\leq \ker\overl\psi$.  On the other hand, $\overl\psi$ factors through $\psi$, hence by Theorem \ref{thm:IsoID}, the image of $\Isom(X)$ factors through the finite group  $\pi_0(\Isom(X))$. Since the kernel of $\Out(\AG)\rightarrow \Aut(H_1(\AG))$ is torsion-free by a result of Day \cite{Day}, $\pi_0(\Isom(X))$ injects into $\Aut(H_1(\AG))$. 
\end{proof}

\section{Maximal tori in blowups}\label{sec:MaxTori}
In this section we will assume the parallelotope structure on $\SP$ is rectilinear.  The blowup structure $\bPi$ induces a labeling on edges of $\SP$ by either generators or partitions, and distinguishes a CAT(0) subcomplex $\EP$ consisting of those cubes labeled only by partitions.  Recall that $\EP$ contains every vertex of $\SP$ and is contractible (being CAT(0)).

\begin{definition} Let $p\colon \USP\rightarrow \SP$ be the universal covering and let $\Delta\subseteq \G$ be a maximal clique.  The \emph{maximal torus} $\T_\Delta$ is the image under $p$ of the min set $\Min(A_\Delta)$.  
\end{definition}
That this min set is nonempty follows from the fact that $A_\Delta$ is free abelian and the flat torus theorem.

\begin{lemma}\label{lem:ExistMaxTorus}For any maximal clique $\Delta$, $\T_\Delta$ is isometric to a torus of dimension $|\Delta|$. Moreover, every cube of $\SP$ is contained in some maximal torus. 
\end{lemma}
\begin{proof} Suppose $\Delta=\{v_1,\cdots,v_k\}$  is a maximal clique and let $C$ be the unique cube in $\SP$ whose edge labels are $\{v_1,\ldots, v_k\}$, whose existence is guaranteed by Corollary 3.7 of \cite{BCV}. By Lemma \ref{lem:BCV3.9}, $C$ lies in the image of the min set of each of the $v_i$.  Since the $v_i$ commute, there exists a lift $\widetilde{C}$ in $\Min(A_\Delta)$. On the one hand, since $A_\Delta$ is free abelian, $\Min(A_\Delta)$ is isometric to $Y\times \R^k$ for some convex set $Y$. On the other hand, since $\widetilde{C}\subset \Min(A_\Delta)$  and $C$ is a maximal cube of $\SP$, this implies $Y$ must be a point, hence $\Min(A_\Delta)\cong \R^k$. Hence $\T_\Delta$ is a torus.  To see that every point $\SP$ lies in such a torus, consider any maximal cube $C$ with edge labels $A_1,\ldots, A_k$ and choose $m_i\in \max(A_i)$ for each $i$.  Then by Lemma \ref{lem:BCV3.9}, each edge of $D$ lies in the image of an axis for some $m_i$.  Arguing as above, we see that $C\subset \T_\Delta$ where $\Delta=\{m_1,\ldots, m_k\}$. 
\end{proof}

Since maximal tori cover $\SP$, it will be useful to understand their intersections. As a first step, we study the intersection of a maximal torus with $\EP$.

\begin{lemma}\label{lem:ConvexIntersection}Each maximal torus intersects $\EP$ in a convex subset.  
\end{lemma} 
\begin{proof}Suppose $\Delta=\{v_1,\ldots, v_k\}$.  From the proof of Lemma \ref{lem:ExistMaxTorus}, $\T_\Delta$ contains the unique cube labeled by $\{v_1,\ldots, v_k\}$.  Choose some vertex $x_0$ in the boundary of this cube as a basepoint and let $\beta_i$ be the image of the axis for $v_i$ passing through $x_0$.  Since $\SP$ is orthogonal, $\T_\Delta\cong \beta_1\times\cdots\times \beta_k$. Now $\beta_i$ contains the edge $e_{v_i}$ labeled by $v_i$ incident at $x_0$.  It follows that if $\gamma_i=\beta_i\setminus \Int(e_{v_i})$, then $\gamma_i\subset \EP$ and from the product structure above we see that $\T_\Delta\cap \EP\cong \gamma_1\times\cdots\times \gamma_k$. This is clearly convex, as it is a product of geodesic segments. 
\end{proof}

Now we consider the intersection of two maximal tori, and show that up to homotopy, they intersect in a torus of smaller dimension, corresponding to the intersection of their maximal cliques. 

\begin{lemma}\label{lem:2ToriMeet}Let $\Delta_1$ and $\Delta_2$ be two maximal cliques such that $\T_{\Delta_1}\cap \T_{\Delta_2}$ contains a vertex.  Then $\T_{\Delta_1}\cap \T_{\Delta_2}=\K\times \T_{\Delta}$ where $\K$ is contractible and $\Delta=\Delta_1\cap\Delta_2$. 
\end{lemma}
\begin{proof}Let $\Delta=\Delta_1\cap \Delta_2$ and set $\Delta_i^\perp=\Delta_i\setminus \Delta$ for $i=1,2$.  For any point $p$ in either $\T_{\Delta_1}$ or $\T_{\Delta_2}$, then there is a unique subtorus $\T_\Delta$ spanned by generators in $\Delta$ containing $p$.  It follows that a translate of $\T_{\Delta}$ passes through every point of  $\T_{\Delta_1}\cap \T_{\Delta_2}$.   Since $\T_{\Delta_i}$ decomposes as an orthogonal product $\T_{\Delta_i}=\T_\Delta\times \T_{{\Delta_i}^\perp}$ for $i=1,2$, we conclude that $\T_{\Delta_1}\cap \T_{\Delta_2}$ decomposes as $\K\times \T_{\Delta}$. Let $x_0\in \T_{\Delta_1}\cap\T_{\Delta_2}$ be a vertex, which exists by assumption. By Lemma \ref{lem:ConvexIntersection},  $\T_{\Delta_1}\cap \EP$ and $\T_{\Delta_2}\cap \EP$ are both convex.  Since $\EP$ contains all the vertices of $\SP$, we see that $ \T_{\Delta_1}\cap\T_{\Delta_2}\cap\EP$ is nonempty and therefore convex in $\EP$. 

To finish the proof it suffices to show that $\K$ is parallel to a subspace of $\EP$. Observe that $\T_{\Delta_i}$ only meets the interiors of cubes with labels in $\Delta_i$ or one of the partitions that split $v\in \Delta_i$, $i=1,2$. In particular, since $\Delta_1\cap \Delta_2=\Delta$, this means that $\T_{\Delta_1}\cap \T_{\Delta_2}$ does not meet the interior of any cube with some label in $\Delta_1^\perp\cup \Delta_2^\perp$.  Recall that the partitions which split distinct generators of $\Delta_i$ are distinct and pairwise linked.  By orthogonality, the copy of $\K$ passing through the common vertex $x_0$ does not meet any cubes with labels in partitions that split $v\in \Delta$.  Therefore it must only meet cubes labeled by partitions, i.e. $\K\subseteq \EP$, as desired.
\end{proof}

\begin{corollary}\label{cor:FullIntersection}Let $\Delta_1$ and $\Delta_2$ be two maximal cliques such that $\T_{\Delta_1}\cap \T_{\Delta_2}$ contains a vertex. Then $\T_{\Delta_1}\cap \T_{\Delta_2}$ is equal to the image in $\SP$ of $\Min(A_{\Delta_1})\cap \Min(A_{\Delta_2})$.
\end{corollary}
\begin{proof}Let $p\colon \USP\rightarrow \SP$ be the universal covering. The inclusion $p(\Min(A_{\Delta_1})\cap \Min(A_{\Delta_2}))\subseteq \T_{\Delta_1}\cap \T_{\Delta_2}$ is clear.  For the converse, since $\K$ is contractible, we have a lift $\K\times \widetilde{\T_\Delta}\rightarrow \Min(A_{\Delta_1})\cap \Min(A_{\Delta_2})$ which is necessarily surjective. 
\end{proof}

The above corollary allows us to define the intersection of two maximal tori purely in terms of CAT(0) geometry and group action.  This will be important in the next section where we consider more general parallelotope structures.  To finish this section, we show that any pair of maximal tori can be connected by a chain of pairwise intersecting maximal tori as in Lemma \ref{lem:2ToriMeet}.

\begin{lemma}Any two maximal tori $\T_1,\T_2$ can be connected by a chain of maximal tori $\T_1=\mbb{M}_0,~ \mbb{M}_2,~\ldots,~ \mbb{M}_{n+1}=\T_2$, such that for  $0\leq i\leq n$, $\mbb{M}_i\cap\mbb{M}_{i+1}$ contains a vertex.  

\end{lemma}

\begin{proof}
By Lemma \ref{lem:ExistMaxTorus}, $\T_1$ contains a maximal cube $C_1$ labeled by the generators in a maximal clique $\Delta_1$, and $\T_2$ contains a cube $C_2$ labeled by generators in a maximal clique $\Delta_2$.  Connect these cubes by a minimal length edgepath $\gamma=e_1\ldots e_n$ in $\EP$.  Applying Lemma \ref{lem:ExistMaxTorus} again, each $e_i$ is contained in a maximal torus $\mbb{M}_i$. Now set $\T_1=\mbb{M}_0$ and $\T_2=\mbb{M}_{n+1}$. Thus, for all $0\leq i\leq n$, we have $\mbb{M}_i\cap\mbb{M}_{i+1}$ contains a vertex along the path $\gamma$. 
\end{proof}

\section{Proof of the main theorem}\label{sec:ProofofMain}%

In this section we prove Theorem \ref{thm:IsoID}. First we must translate the results of the previous section on maximal tori to the setting of skewed $\G$-complexes.  This is possible because of characterization of maximal tori and their intersections in terms of min sets.

\begin{lemma}\label{lem:MinPreserve}Let $(\SP, \F)$ be a skewed metric blowup and let $\sigma\colon(\SP,\F)\rightarrow (\SP,\E)$ be the straightening map. For any maximal clique $\Delta\subseteq \Gamma $, the straightening map $\sigma$ takes the maximal torus for $\Delta$ in $(\SP,\F)$ to a maximal torus for $\Delta$ in $(\SP,\E)$.
\end{lemma}
\begin{proof} Our identification of the fundamental group of $(\SP,\F)$ with $\AG$ comes from the composition of the straightening map $\sigma\colon (\SP,\F)\rightarrow (\SP,\E)$ with the collapse map $c_\pi\colon (\SP,\E)\rightarrow \SaG$. Hence the straightening map $\sigma$ is equivariant with respect to the actions of $\AG$ on universal cover of $(\SP, \F)$ and $(\SP,\E)$.  It follows that $\sigma$ takes the min set of $A_\Delta$ in the universal cover of $(\SP, \F)$ to the corresponding min set in the universal cover of $(\SP,\E)$.  
\end{proof}

It follows from Lemma \ref{lem:MinPreserve} that all of the results obtained in Section \S\ref{sec:MaxTori} still hold for $(\SP,d)$.  For convenience, we record them in a corollary.

\begin{corollary}\label{cor:SameFacts}Let $(\SP,d)$ be a skewed metric blowup.  
\begin{enumerate}
\item For any maximal clique $\Delta$, there is a torus $\T_\Delta\subseteq \SP$ with a flat metric of dimension $|\Delta|$.  Every parallelotope of $(\SP,\F)$ is contained in some maximal torus.
\item Let $\Delta_1$ and $\Delta_2$ be two maximal cliques such that $\T_{\Delta_1}\cap \T_{\Delta_2}$ contains a vertex.  Then $\T_{\Delta_1}\cap \T_{\Delta_2}=\K\times \T_{\Delta}$ where $\K$ is convex and $\Delta=\Delta_1\cap\Delta_2$. 
\item Any two maximal tori $\T_1,\T_2$ can be connected by a chain of maximal tori $\T_1=\mbb{M}_0,~ \mbb{M}_2,~\ldots,~ \mbb{M}_{n+1}=\T_2$, such that for  $0\leq i\leq n$, $\mbb{M}_i\cap\mbb{M}_{i+1}$ contains a vertex.  
\end{enumerate}
\end{corollary}
\begin{rmk} By Corollary \ref{cor:FullIntersection}, $\T_{\Delta_1}\cap \T_{\Delta_2}$ is the image of the convex set $\Min(A_{\Delta_1})\cap \Min(A_{\Delta_2})$, which decomposes as an orthogonal product $\K\times\T_{\Delta}$.  However, as $\EP$ need not be preserved under straightening, the $\K$ appearing here is not necessarily the same $\K$ as in Lemma \ref{lem:2ToriMeet}.
\end{rmk}
The next lemma is a direct consequence of the existence of maximal tori.  
\begin{lemma}\label{lem:RestrictTranslation} If $f$ is homotopic to the identity, then $f$ restricts to a translation on each maximal torus of $\SP$.  
\end{lemma}
\begin{proof}Lift $f$ to an isometry $\widetilde{f}$ from the universal cover $\USP$ to itself.  Since $f$ is homotopic to the identity, $\widetilde{f}$ is equivariantly homotopic to the identity. As $\widetilde{f}$ is an isometry, it therefore preserves the min set of every element, as well as the min set of every abelian special subgroup. It follows that $f$ maps every maximal torus $\T_\Delta$ to itself. Since it is homotopic to the identity, $f$ restricts to a translation on each $\T_\Delta$.
\end{proof}

Let $\tau$ be a translation on a torus $\T$.  Each such $\tau$  is contained in a unique closed subgroup which we denote $\mbb{G}_\tau$, and a unique closed connected subgroup which we denote $\T_\tau$.  Since $\mbb{G}_\tau$ is the closure of a cyclic subgroup, it follows from general theory of compact Lie groups that $\mbb{G}_\tau$ is isomorphic to $\Sa_\tau\times \mbb{F}$ where $\Sa_\tau$ is a closed connected subgroup of $\T$, and $\mbb{F}\cong \Z/n\Z$ is finite cyclic. The finite factor $\mbb{F}$ lies on a unique 1-parameter subgroup $\mbb{L}_\tau$. Hence $\T_\tau=\Sa_\tau\times\mbb{L}_\tau$ and $\pi_1(\T_\tau)=\pi_1(\Sa_\tau)\times \pi_1(\mbb{L}_\tau)$. In order to understand how an isometry as in Lemma \ref{lem:RestrictTranslation} restricts to intersections of maximal tori we will need the following lemma about which translations preserve connected subsets of a torus. 

\begin{lemma} \label{lem:ClosedOrbit}Let $\T$ be a torus, and $\mathbb{X}\subseteq \T$ a closed, connected subset.  If $\mbb{X}$ is preserved by some translation $\tau$, then the image of $\pi_1(\mbb{X})$ under inclusion contains that of  $\pi_1(\T_\tau)$.

\end{lemma}
\begin{proof} Pick any $x_0\in \mbb{X}$ and consider the translate $\mbb{X}-x_0$. Since $\mbb{X}$ is preserved by $\tau$, so is $\mbb{X}-x_0$. Without loss of generality, we therefore assume $\mbb{X}$ contains $0$.  As $\mbb{X}$ is closed, $\mbb{X}$ contains $\mbb{G}_\tau=\Sa_\tau\times \mbb{F}$.  Let $g$ be a generator of $\mbb{F}\cong \Z/n\Z$.  Since $\mbb{X}$ is connected, we can find a path in $\mbb{X}$ from 0 to $g$.  Translating by powers of $g$, we can find a loop in $\mbb{X}$ homotopic to $\mbb{L}_\tau$.  Thus, the image of $\pi_1(\mbb{X})$ contains the image of $\pi_1(\Sa_\tau)\times \pi_1(\mbb{L}_\tau)=\pi_1(\T_\tau)$.
%
\end{proof}

We are now ready to prove Theorem \ref{thm:IsoID}.

\begin{proof}[Proof of Theorem \ref{thm:IsoID}]
Let $f:\SP\rightarrow \SP$ be an isometry that is homotopic to the identity. By Lemma \ref{lem:RestrictTranslation}, $f$ restricts to a translation on each maximal torus.  Suppose $\T_{\Delta_1}$ and $\T_{\Delta_2}$ are two maximal tori whose intersection contains a vertex.  Then $f|_{\T_{\Delta_1}}=\tau_1$ and $f|_{\T_{\Delta_2}}=\tau_2$ for some translations $\tau_1$ and $\tau_2$. By Corollary \ref{cor:SameFacts}(2), the image of $\pi_1(\T_{\Delta_1}\cap\T_{\Delta_2})$ in either $\pi_1(\T_{\Delta_1})$ or $\pi_1(\T_{\Delta_2})$ is just $\pi_1(\T_{\Delta_1\cap\Delta_2})$.  Hence by  Lemma \ref{lem:ClosedOrbit}, $\tau_1$ and $\tau_2$ each belong to $\T_{\Delta_1\cap\Delta_2}$.  Since the intersection of all maximal cliques in $\Gamma$ is $\Cent(\Gamma)$, by Corollary \ref{cor:SameFacts}(3), the fact that any two maximal tori can be connected by a chain implies that $f$ restricts to a translation by $\tau\in \T_{\Cent(\Gamma)}$ on each maximal torus. Now consider the isometry $g=\tau^{-1}\circ f$. Then $g$ restricts to the identity on every maximal torus.  By Corollary \ref{cor:SameFacts}(1), we conclude that $g=\id$ and hence $f=\tau$ is a translation in the central torus. 
%
%
%
\end{proof}

%
%
%

\bibliographystyle{plain}
\bibliography{IsomBib}

\end{document}